\documentclass[11pt,reqno]{amsart}

\usepackage{graphicx}
\usepackage{color}
\usepackage{bm}
\usepackage{amsmath, amssymb, amsthm}
\usepackage{hyperref}
\usepackage[margin=1.5in]{geometry}

\renewcommand{\div}{\nabla\cdot}
\newcommand{\grad}{\nabla}

\newcommand{\Br}{{\bf x}}

\newcommand{\Bu}{{\bf u}}

\newcommand{\Bx}{{\bf x}}

\newcommand{\Bk}{{\bf k}}
\newcommand{\Bv}{{\bf v}}

\newcommand{\BA}{{\bf A}}

\newcommand{\BR}{{\bf R}}
\newcommand{\BE}{{\bf E}}

\newcommand{\BJ}{{\bf J}}
\renewcommand{\BR}{{\bf R}}

 \newcommand{\bb}{\mathbf b}

\newcommand{\bk}{\mathbf k} 

 \newcommand{\bx}{\mathbf x}

 \newcommand{\bJ}{\mathbf J}

\DeclareMathOperator*{\argmin}{arg\,min}

\newcommand{\YZ}[1]{\textcolor{black}{#1} }
\newcommand{\WL}[1]{\textcolor{black}{#1} }

\newtheorem{thm}{Theorem}

\newtheorem{prop}[thm]{Proposition}

\newtheorem{coro}[thm]{Corollary}
\newtheorem{remark}[thm]{Remark}
\newtheorem{hypo}[thm]{Hypothesis}

\begin{document}

\title[Acousto-electric Inverse Source Problems]{Acousto-electric Inverse Source Problem
}

\author{Wei Li, John C. Schotland, Yang Yang, Yimin Zhong}
\address{Department of Mathematical Sciences, DePaul University, Chicago, IL 60604}
\email{wei.li@depaul.edu}
\address{Department of Mathematics, Yale University, New Haven, CT 06511}
\email{john.schotland@yale.edu}
\address{Department of Computational Mathematics, Science and Engineering, Michigan State University, East Lansing, MI 48824}
\email{yangy5@msu.edu}
\address{Department of Mathematics, Duke University, Durham, NC 27708}
\email{yimin.zhong@duke.edu}

\date{\today}

\begin{abstract}
We propose a method to reconstruct the electrical current density inside a conducting medium from acoustically-modulated boundary measurements of the electric potential. We show that the current can be uniquely reconstructed with Lipschitz stability. We also perform numerical simulations  to illustrate the analytical results, and explore the partial data setting when measurements are taken only on part of the boundary.
\end{abstract}

\maketitle

\section{Introduction}

Electroencephalography is widely used in neurology and neurosurgery to monitor the electrical activity of the human brain~\cite{Cohen2014,Shorvon2013,Baltuch2020}. In a typical clinical setting, the electrical signal is recorded from electrodes that are placed either on the scalp or surgically implanted in the brain. In either case, the objective is to locate and characterize the current source that produces the measured signal. An important application is to the localization of seizure foci 
in patients undergoing epilepsy surgery. In mathematical terms, this problem is closely related to the inverse problem of reconstructing the electrical current density of a conducting medium from boundary measurements. It is well known, however, that the inverse source problem is underdetermined and does not admit a unique solution~\cite{Bleistein_1977, Dassios_2005, Devaney_1973,Fokas_1996,Albanese_2006}. That is, more than one source gives rise to the same measurements. This problem may be overcome, to some extent, if {\em a priori} information about the source is known. For instance, if the source consists of a single current dipole (or even a fixed number of dipoles), then its position and strength can be uniquely determined {\cite{grech2008review,schimpf2002dipole,Ammari_2002,zhao2018inverse}}. However, electrical activity in the brain is distributed across networks of neurons of unknown structure, and thus this state of affairs is highly unsatisfactory

In this work we consider an alternative approach to the inverse source problem which, in some sense, is in the spirt of several recently proposed hybrid imaging modalities~\cite{Ammari_2008,Bal_2010, Capdeboscq_2009,Gebauer_2008,Kuchment_2011,Nachman,Triki_2010}. In these methods, a wavefield is used to control the material properties of a medium of interest, which is then probed by a second wavefield~\cite{Bal_2010,Bal_2014,Bal_2016,Chung_2017,Chung_2020_1,
Chung_2020_2,Schotland_2020,kuchment2012stabilizing, kempe1997acousto, 
Chung2021, li2019hybrid,li2020inverse}. Here we exploit the {acousto-electric} effect, in which the density of charge carriers and conductivity are spatially modulated by an acoustic wave~\cite{Lavandier_2000,Parmenter_1953}. We find that it is possible to {\em uniquely} recover the current density from  boundary measurements of the electrical potential. Moreover, the stability of the reconstruction is shown to be Lipschitz, which provides mathematical justification for the use of acoustic modulation in the electrical inverse source problem.

The remainder of this paper is organized as follows. In Section~2 we introduce a model for the acousto-electric effect. This model is used as the basis for the treatment of the acoustically-modulated inverse source problem that is taken up in Section~3. We show that the boundary measurements in the presence of acoustic modulation lead to knowledge of an internal functional, from which the current source may be recovered. In Section~4 our results are illustrated by numerical simulations , including the cases of full and partial boundary measurements, along with an alternating minimization algorithm that improves numerical stability.  Finally, our conclusions are presented in Section~5.

\section{Acousto-electric effect}
We begin by developing a simple model for the acousto-electric effect, following the approach of~\cite{Bal_2010}.  Consider a system of conducting particles and charge carriers in a fluid. If a small-amplitude acoustic wave is incident on the system, each particle will oscillate about its local equilibrium position. We may thus regard the particles as independent. It follows
that the equation of motion of a single particle is of the form
\begin{equation}
\label{eq_motion_u}
\rho \frac{d \Bu}{dt} =  -\grad p \ ,
\end{equation}
where $\Bu$ is the velocity of the particle, $\rho$ is its mass density, and $p$ is the pressure in the fluid. We consider a standing time-harmonic plane wave of frequency $\omega$ with
\begin{equation}
p = A\cos(\omega t) \cos(\Bk\cdot\Br + \varphi) \ ,
\end{equation}
where $A$ is the amplitude of the wave, $\Bk$ is its wave vector and $\varphi$ is the phase.
For simplicity, we have assumed that the speed of sound $c_s$ is constant with {$|\Bk|=\omega/c_s$}. The oscillatory solution to (\ref{eq_motion_u}) is given by
\begin{equation}
\Bu =  \frac{A }{\rho\omega} \sin(\omega t) \sin(\Bk\cdot\Br + \varphi) \Bk \ .
\end{equation}
Thus apart from a transient, the particle moves with the fluid.

In the presence of an applied field, the charge carriers move and generate a current. The current density $\BJ_\epsilon$ is of the form
\begin{equation}
\BJ_\epsilon = \sum_i q_i \Bv_i \delta(\Br -\BR_i(t)) \ ,
\end{equation}
where $\BR_i$ is the position of the $i$th charge carrier, $\Bv_i$ is its velocity and $q_i$ is the charge. Since each particle is independent, it follows from integration
of the equations of motion (\ref{eq_motion_u}), that $\BJ_\epsilon$ is given by
\begin{equation}
\label{def_J}
\BJ_\epsilon(\Br)=\BJ_0(\Br)\left[1+ \epsilon\cos(\Bk\cdot\Br + \varphi)\right] \ ,
\end{equation}
where $\BJ_0$ is the current in the absence of the acoustic wave and $\epsilon=A/(\rho c_s^2)\ll 1$ is a small parameter. 
The conductivity $\sigma_\epsilon$ of the medium is proportional to the density of conducting particles and is given by 
\begin{equation}
\label{def_sigma}
\sigma_\epsilon(\Br) = \sigma_0(\Br)\left[1+ \epsilon\beta\cos(\Bk\cdot\Br + \varphi)\right] \ ,
\end{equation}
where $\sigma_0$ is the unmodulated conductivity and $\beta$ is the zero-frequency elasto-electic constant. 
We conclude that the acoustic wave leads to spatial modulation of the current and the conductivity.

Consider the flow of current in a bounded domain $\Omega \subset \mathbb{R}^n$ with a smooth boundary, $n \geq 2$. The total current 
\begin{equation}
\BJ=\BJ_\epsilon + \sigma_\epsilon \BE \ ,
\end{equation}
consists of contributions from the source and the volume, where $\BE$ is the electric field. Under static conditions, the conservation of charge takes the form
$\div\BJ=0$. In addition, $\BE=-\grad u$, where $u$ is the potential. The potential then obeys the equation
\begin{eqnarray}
\label{pde}
\div\sigma_\epsilon(\Br)\grad u_\epsilon &=& \div\BJ_\epsilon \quad {\rm in \quad \Omega} \ , \\
\frac{\partial u_\epsilon}{\partial n} &=& 0 \quad {\rm on \quad \partial\Omega} \ ,
\label{bc_u}
\end{eqnarray}
where the Neumann boundary condition prevents the outward flow of current through $\partial\Omega$. 

We now turn to the derivation of an internal functional from boundary measurements of the potential. In Section~3 we will show that it is possible to recover the current source from the internal functional. The following assumptions are imposed throughout the paper: 
\begin{enumerate}
	\item [(A-1)] The domain $\Omega$ is simply connected with $C^2$ boundary $\partial\Omega$.
	\item [(A-2)] The (unmodulated) conductivity $\sigma_0\in L^\infty(\Omega)$ is known and satisfies
	\begin{equation}\label{eq:elliptic}
	0<K_1<\sigma_0<K_2 \ ,
	\end{equation}
	for some positive constants $K_1$ and $K_2$. 
        \item[(A-3)] $\BJ_0\in (L^2(\Omega))^n$ and $\BJ_0$ is compactly supported in $\Omega$. 
\end{enumerate}
Under these assumptions, the boundary value problem~\eqref{pde} admits a unique weak solution $u_\epsilon\in H^1(\Omega)$ up to an additive constant~\cite{gilbarg2015elliptic}, satisfying 
\begin{equation}
\label{eq:reg}
\|\nabla u_\epsilon\|_{L^2(\Omega)} \leq \frac{1}{K_1} \|\BJ_\epsilon\|_{\WL{(L^2(\Omega))^n}}\,.
\end{equation}
We thus find that \WL{$\BJ_\epsilon \in (L^2(\Omega))^n$}. 

To derive the internal functional, we consider the following auxiliary boundary value problem
\begin{equation}
\label{eq:eqnv}
\begin{aligned}
    \div\sigma_0(\Br)\grad v_j &= 0  \quad &&{\rm in \quad \Omega} \,, \\
\frac{\partial v_j}{\partial n} &= g_j \quad &&{\rm on \quad \partial\Omega} \,,
\end{aligned}
\end{equation}
where $g_j\in H^{-1/2}(\partial\Omega)$, $j=1,\ldots, N$ are prescribed boundary sources. Under the assumptions (A-1) and (A-2), this auxiliary boundary value problem admits a unique weak solution $v_j\in H^1(\Omega)$ 
up to an additive constant. Since the unmodulated conductivity $\sigma_0$ is known, the solutions $v_j$ in principle can be computed. 

Next, multiplying (\ref{pde}) by $v_j$ and (\ref{eq:eqnv}) by $u_\epsilon$,
subtracting the resulting equations and integrating the difference over $\Omega$ yields
\begin{equation}
\label{identity}
\Sigma_\epsilon^{(j)} = \int_\Omega  \left[(\sigma_\epsilon - \sigma_0)\grad u_\epsilon \cdot \grad v_j  +  v_j \div \BJ_\epsilon \right] dx
\end{equation}
Here the surface
term $\Sigma_\epsilon^{(j)}$, which follows from an integration by parts, is defined by
\begin{equation}
\label{eq:bdy}
\Sigma_\epsilon^{(j)} := \int_{\partial\Omega}  \left[ u_\epsilon \sigma_0 \frac{\partial v_j}{\partial n}   
-v_j  \sigma_\epsilon \frac{\partial u_\epsilon}{\partial n}  \right] dx \ .
\end{equation}
Making use of the boundary conditions  (\ref{bc_u}) and  (\ref{eq:eqnv}), we see that
\begin{equation}
\label{eq:bdy_full}
\Sigma_\epsilon^{(j)} = \int_{\partial\Omega}  u_\epsilon \sigma_0 g_j   dx \ .
\end{equation}
Therefore $\Sigma_\epsilon^{(j)}$ can be determined from boundary measurement of $u_\epsilon$. 

We now introduce the asymptotic expansions for $u_\epsilon$ and $\Sigma^{(j)}_\epsilon$ as
\begin{eqnarray}
u_\epsilon &=& u_0 + \epsilon u_1  + \cdots \ ,\\
\Sigma^{(j)}_\epsilon &=& \Sigma^{(j)}_0 + \epsilon \Sigma^{(j)}_1 + \cdots \ ,
\end{eqnarray}
which we substitute into (\ref{identity}).
At $\mathcal{O}(1)$ we obtain
\begin{equation}
    \begin{aligned}
        \Sigma^{(j)}_0 = \int_\Omega   \ v_j \div \BJ_0  \, dx\,,    
    \end{aligned}
\end{equation}
and at $\mathcal{O}(\epsilon)$ we have
\begin{equation}
\label{fourier}
\begin{aligned}
    \Sigma^{(j)}_1 =\int_\Omega  \left(\beta\sigma_0 \grad u_0 - \BJ_0 \right)\cdot\grad v_j \cos(\Bk\cdot 
    \Bx + \varphi) \, dx .
\end{aligned}
\end{equation}
Here we have inserted (\ref{def_J}) into~\eqref{identity}, performed a further integration by parts, and then applied the assumption that $\BJ_0$ vanishes on $\partial\Omega$. Since $\Sigma^{(j)}_\epsilon$ is determined by the boundary measurement, $\Sigma^{(j)}_1$ is known. Provided the experiment is repeated with different $\bk$ and $\varphi$, it follows from
(\ref{fourier}) that by inverting a Fourier transform, we can recover the  internal functional
\begin{equation}
\label{eq:int_scalar}
H_j := \left(\beta\sigma_0 \grad u_0 - \BJ_0 \right)\cdot\grad v_j 
\end{equation}
at every point in $\Omega$.
\begin{remark}
Despite the fact that the solutions $u_\epsilon$ and $v_j$ are known up to an additive constant, the internal functional $H_j$ is unique.
\end{remark}

\section{Inverse Problem}
\label{sec:inv}
It follows from the above, that the inverse problem consists of recovering the (unmodulated) source current $\BJ_0$ from the internal functional $H_j$. In this section we will derive a reconstruction procedure that uniquely recovers $\BJ_0$ with Lipschitz stability. The following hypothesis is necessary throughout:
\begin{hypo}\label{lem:indep}
There exist $g_j$, $j=1,\cdots,N$ such that the gradients of the solutions  $v_j \in C^1(\overline{\Omega})$ to the auxiliary problems~\eqref{eq:eqnv} form a basis everywhere in $\Omega$. That is
\begin{equation}\label{eq:indep}
\det[\grad v_1,\cdots, \grad v_N] \neq 0 \quad {\rm in} \quad \overline{\Omega} \ .
\end{equation}
\end{hypo}
This hypothesis holds at least for sufficiently regular conductivity $\sigma_0$. 

\begin{prop}
Let  $N\geq2$. Under the assumptions (A-1) and (A-2), if $\sigma_0\in H^{\frac{N}{2} + 2 +s}(\Omega)$ for some $s>0$, there exist $N$  (complex-valued) solutions $v_1,\cdots v_N\in C^1(\overline{\Omega})$ to \eqref{eq:eqnv} with pointwise linearly independent gradients.
\end{prop}
\begin{proof}
It has been established in~\cite[Lemma 2.1]{BalGuoMonard2014} that for $N\geq 2$, the equation~\eqref{eq:eqnv} admits complex-valued solutions of the form
\begin{equation}\label{eq:CGO}
v_j(\bx;\rho_j) = \frac{1}{\sqrt{\sigma_0(\bx)}}e^{\rho_j\cdot \bx}(1+\psi_{\rho_j}(\bx) )\,,
\end{equation}
where $\rho_j\in\mathbb{C}^N$ is a complex parameter with $\rho_j \cdot\rho_j=0$ and the function $\psi_{\rho_j}$ satisfies the estimate
\begin{equation}\label{eq:corrector}
|\rho_j| \|\psi_{\rho_j}\|_{C^0(\overline \Omega)} + \|\psi_{\rho_j}\|_{C^{1}(\overline\Omega)} \leq C\left\|\frac{\Delta\sqrt{\sigma_0}}{\sqrt{\sigma_0}}\right\|_{H^{\frac{N}{2}+s}(\Omega)}
\end{equation}
for some constant $C = C(\Omega, s )> 0$. The right-hand-side of~\eqref{eq:corrector} is finite as a result of the assumptions (A-2) and $\sigma_0\in H^{\frac{N}{2} + 2 +s}(\Omega)$. Observe that the gradient of $v_j$ is
$$
\nabla v_j = \frac{e^{\rho_j \cdot \bx}}{\sqrt{\sigma_0}}\left(\rho_j +\rho_j \psi_{\rho_j}+\nabla\psi_{\rho_j}-\frac{1}{2}(1+\psi_{\rho_j})\sigma_0^{-\frac{3}{2}}\nabla \sigma_0\right),
$$
then
$$
\begin{aligned}
\det\left[ \nabla v_1,\cdots, \nabla v_N \right] &= \prod_{j=1}^N|\rho_j|\frac{\exp({\sum_{j=1}^N \rho_j\cdot \bx})}{\sigma_0^{N/2}}\\
\times\Big( &\det\left\{  \frac{\rho_1}{|\rho_1|},\cdots,  \frac{\rho_N}{|\rho_N|} \right\} + O\big(\max_{1\leq j\leq N}\left\{\frac{1}{|\rho_j|}\right\}\big) \Big).
\end{aligned}
$$
In particular, if we choose vectors $\rho_j$ as
\begin{equation}
\begin{aligned}
\rho_{2m-1} &= \frac{\sqrt{2}}{2} |\rho_{2m-1}|( e_{2m-1}+ \textit{i}\, e_{2m}),\\ \rho_{2m} &= \frac{\sqrt{2}}{2}|\rho_{2m}|(e_{2m-1} - \textit{i}\,  e_{2m}),
\end{aligned}
\end{equation}
for $m=1,\cdots, \lfloor \frac{N }{2}\rfloor$, where $\textit{i}= \sqrt{-1}$ is the imaginary unit and $e_j$ denotes the unit vector whose $j$th component is $1$ and other components are $0$, then the determinant $\det\left[ \nabla v_1,\cdots,\nabla v_N \right]$ is bounded away from zero uniformly when $\min_{1\le j\le N}|\rho_j|$ is sufficiently large. 
{This is true because the matrix $\left[\frac{\rho_1}{|\rho_1|},\cdots,  \frac{\rho_N}{|\rho_N|}  \right]$ is blockwise diagonal with blocks of the form
$$
A = \begin{bmatrix}1&\textit{i}\\1&-\textit{i} \end{bmatrix}, \quad B = \begin{bmatrix}1&\textit{i}&0\\1&-\textit{i} &0\\0&1&\textit{i}\end{bmatrix}\ .
$$
If $N$ is even, then the matrix contains $\frac{N}{2}$ blocks of $A$, and if $N$ is odd, then the matrix has $(\lfloor\frac{N}{2}\rfloor - 1))$ blocks of $A$ and one block of $B$. Since $\det(A) = -2\textit{i}$, $\det(B) = 2$, we obtain $\big|\det\left\{  \frac{\rho_1}{|\rho_1|},\cdots,  \frac{\rho_N}{|\rho_N|} \right\}\big| \geq 2$. }
The boundary potential sources $g_j$ in Hypothesis~\ref{lem:indep} then can be taken as $g_j := \partial_n v_j|_{\partial\Omega}$, $j=1,\dots,N$.

\end{proof}

Let $v_1, \dots, v_N$ be the auxiliary solutions in Hypothesis~\ref{lem:indep} with linearly independent gradients, and let $H_j$ be the internal functional corresponding to $v_j$ as in~\eqref{eq:int_scalar}, $j=1,\dots,N$. Then 
$$
[H_1, \dots, H_N] = (\beta\sigma_0 \nabla u_0 - \bJ_0) [\nabla v_1, \dots, \nabla v_N],
$$
where $(\beta\sigma_0 \nabla u_0 - \bJ_0)$ is viewed as a row vector. If we set
\begin{equation}\label{eq:HtoA}
\BA := [H_1, \dots, H_N] [\nabla v_1, \dots, \nabla v_N]^{-1},
\end{equation}
then 
\begin{equation}
\label{eq:internal}
\BA = \beta\sigma_0 \grad u_0 - \BJ_0\ .
\end{equation}
Since each $H_j$ is known from  boundary measurements and $v_j$ can be obtained via solving the auxiliary problem~\eqref{eq:eqnv} with the prescribed boundary potential source $g_j$, we can compute the matrix $\BA$ explicitly.
If $\beta\neq 1$, by taking the divergence of~\eqref{eq:internal} and combining the result with the equation~\eqref{pde} for $\epsilon=0$, we find that
\begin{equation}
\label{eq:JinA}
\div \BJ_0 = \frac{1}{\beta -1} \div \BA \,.
\end{equation}
Thus we can solve for $u_0$ up to an additive constant from the boundary value problem~\eqref{pde}, and then compute $\BJ_0$ using
\begin{equation}
\label{eq:recJ_0}
\BJ_0 = \beta\sigma_0\grad u_0 - \BA \,.
\end{equation}
Note that $\BJ_0$ is uniquely determined, since $u_0$ is unique up to a constant. Evidently the above procedure breaks down if $\beta =1$.

Finally, we show that the reconstruction from the internal functional $H_j$, $j=1,\dots,N$, has Lipschitz stability.

\begin{thm}\label{prop:stability}
The reconstruction \eqref{eq:recJ_0} is Lipschitz stable in the sense that if $\BJ_0$ and $\tilde{\BJ}_0$ are currents reconstructed from the corresponding internal functionals $\BA$ and $\tilde \BA$, then
\begin{equation}
\|{\BJ}_0- \tilde{\BJ}_0\|_{\WL{(L^2(\Omega))^n}} \leq \left(1 + \frac{|\beta|K_2}{|\beta - 1|K_1}\right) \|\BA-\tilde{\BA}\|_{\WL{(L^2(\Omega))^n}} .
\end{equation}
\end{thm}

\begin{proof}
The stability estimate is easily seen in two steps. First, from (\ref{eq:internal}),
\begin{equation}
\|{\BJ}_0- \tilde{\BJ}_0\|_{\WL{(L^2(\Omega))^n}} \leq  \|\BA-\tilde{\BA}\|_{\WL{(L^2(\Omega))^n}} + |\beta| K_2 \|\nabla(u_0-\tilde u_0)\|_{\WL{(L^2(\Omega))^n}} \ .
\end{equation}
Second, combining \eqref{eq:JinA} and \eqref{pde}, we have that $u_0 - \tilde{u}_0$ satisfies
\begin{equation}
    \begin{aligned}
        \div\sigma_0(\Br)\grad (u_0 - \tilde{u}_0) &=  \frac{1}{\beta -1} \div ( \BA - \tilde{\BA}) \quad &&{\rm in \quad \Omega} \,, \\
        \frac{\partial (u_0 - \tilde{u}_0) }{\partial n} &= 0 \quad &&{\rm on \quad \partial\Omega} \,. 
        \end{aligned}
\end{equation}
Similar to the regularity property \eqref{eq:reg}, 
\begin{equation}
\|u_0-\tilde u_0\|_{H^1(\Omega)} \leq \frac{1}{|\beta - 1|K_1} \|{\BA}- \tilde{\BA}\|_{\WL{(L^2(\Omega))^n}} \ , 
\end{equation}
which leads to the conclusion.
\end{proof}

The above stability result can be restated in terms of the internal functionals $H_j$ as follows.
\begin{coro}\label{lem:stabilityH}
Let $\BJ_0$, $\tilde{\BJ}_0$, $\BA$ and $\tilde \BA$ be as in Theorem~\ref{prop:stability}, and let $H_j$ and $\tilde H_j$, $j=1,\cdots N$, be the corresponding scalar internal functions. If Hypothesis \ref{lem:indep} holds, then there exists a constant $C$, such that 
\begin{equation}
\|{\BJ}_0- \tilde{\BJ}_0\|^2_{\WL{(L^2(\Omega))^n}} \leq C \sum_{j=1}^N \|H_j-\tilde{H_j}\|^2_{L^2(\Omega)} \ .
\end{equation}
\end{coro}
\begin{proof}
When Hypothesis \ref{lem:indep} holds, there exists a constant $C_1$ such that
$$
|[\nabla v_1, \dots, \nabla v_N]^{-1}|_2 \leq C_1 \quad \text{for every}\quad \Bx\in\Omega\ ,
$$
where $|\cdot|_2$ is the matrix 2-norm. Thus combining \eqref{eq:HtoA} and Theorem~\ref{prop:stability}, the proof is finished.
\end{proof}

\section{Numerical Reconstruction and Validation}
\label{sec:num}
In this section we present numerical experiments to validate the proposed reconstruction procedure for $\BJ_0$. Reconstructions from both full and partial boundary measurements are reported. 


\subsection{Forward Problem}
The forward problem \eqref{pde} is solved to generate simulated measurements.
Existence and uniqueness of the solution $u_0$, up to a constant, is ensured by standard elliptic theory~\cite{gilbarg2015elliptic}. This boundary value problem is numerically solved with the \YZ{first order Lagrangian finite element method}. The finite element discretization results in a linear system of the form
$A\bx = \bb$ with $\mathrm{Ker} A = \mathrm{Span}\{\mathbf{1}\}$, where $\mathbf{1}$ is the vector whose components are all $1$'s. This linear system is then solved with the \YZ{biconjugate gradient stabilized method (BICGSTAB)}~\cite{gutknecht1993variants} by projecting the discretized solution onto the orthogonal complement $(\mathrm{Ker} A)^{\perp} = \mathbf{1}^{\perp}$.

\subsection{Measurements}
The measurements consist of the internal functionals $H_j$ in \eqref{eq:int_scalar}. To proceed, we must first compute the $v_j$ which solve \eqref{eq:eqnv} for $j=1,2$. 
The boundary sources $g_j$ must be selected so that Hypothesis \ref{lem:indep} holds. That is, 
\begin{equation}
    \det \begin{bmatrix}
    \nabla v_1 ,\nabla v_2
    \end{bmatrix} (\bx) \neq 0  \ , \quad\quad  \bx\in\Omega.
\end{equation}
Since measurements may carry noise, we would like to choose $g_1$ and $g_2$ so that the condition number of the above matrix remains small, in order to achieve stable numerical reconstruction. For example, if $g_1$ and $g_2$ are chosen such that $\|\nabla v_1\|=\|\nabla v_2\|=1$, then simple computation shows the matrix has the smallest condition number when 
$\nabla v_1 \perp \nabla v_2$. 
In the special case where $\sigma_0$ is constant, 
we can simply take the linear functions $v_j(\bx) = \mathbf{d}_j\cdot \bx$ with $\mathbf{d}_1\perp \mathbf{d}_2$.

\subsection{Optimization}
For non-constant $\sigma_0$, $g_j$ can be selected by solving the following minimax problem:
\begin{equation}
    (g_1^{\ast}, g_2^{\ast}) = \argmin_{g_1, g_2\in G} \max_{x\in\Omega} \Big|\frac{\nabla v_1}{|\nabla v_1|}\cdot \frac{\nabla v_2}{|\nabla v_2|}\Big|\,,
\end{equation}
where $G = \{g\in L^2(\partial\Omega) : \|g\|_{L^2(\partial\Omega)} = 1\}$. However, this \YZ{minimax} problem is not in the convex-concave setting~\cite{nemirovski2004prox} and cannot not be efficiently solved by minimizing the primal-dual gap~\cite{komiya1988elementary}. 
Instead, we relax the minimax problem to the following alternating minimization problem.

Suppose the medium permits a solution $v^0_1$ such that $\nabla v^0_1\neq 0$ everywhere in $\Omega$. Then we iteratively take alternating minimization steps. At the $k$th iteration, we solve
\begin{equation}
g_2^k = \argmin_{g_2^k\in G}\frac{1}{2}\int_{\Omega} |\nabla\phi \cdot \nabla v_1^{k-1} |^2 dx \quad \text{subject to}\quad\begin{cases}
\nabla\cdot \sigma_0\nabla\phi = 0\quad&\text{ in }\Omega\,,\\ 
\frac{\partial \phi}{\partial n} = g_2^k\quad &\text{ on } \partial \Omega\,.
\end{cases}
\end{equation}
Next, we set $v_2^k = \phi$ and solve
\begin{equation}
g_1^k = \argmin_{g_1^k\in G}\frac{1}{2}\int_{\Omega} |\nabla\psi \cdot \nabla v_2^{k} |^2 dx \quad \text{subject to}\quad\begin{cases}
\nabla\cdot \sigma_0\nabla\psi = 0\quad&\text{ in }\Omega\,,\\ 
\displaystyle\frac{\partial \psi}{\partial n} = g_1^k\quad &\text{ on } \partial \Omega\,,
\end{cases}
\end{equation}
and we set $v_1^k = \psi$. The iteration is terminated when either the increments in $v_1, v_2$ are smaller than a prescribed tolerance or the maximum iteration number is reached.

The above minimization consists of two convex quadratically constrained quadratic programs and we can apply the \YZ{interior point method~\cite{byrd1999interior, waltz2006interior}} to solve them.
Although the solution to this relaxed alternating minimization problem is not necessarily the solution to the original minimax problem, the boundary conditions selected from the relaxed problem do stabilize the numerical reconstruction, as is shown in the subsequent numerical examples.

\subsection{Numerical examples}
We present numerical examples with full boundary measurement in this subsection, and partial boundary measurement in the next subsection.
The Shepp-Logan head phantom for $\sigma_0$ in the rectangular computational domain $\Omega = [0.1, 0.9]\times [0, 1]$ is used to model the anatomy in an experiment, as shown in Fig~\ref{conductivity}. The values of $\sigma_0$ are assigned based on the conductivities of white matter, grey matter, cerebrospinal fluid and bone, 
where the region outside of the skull is assumed to have the same conductivity as the scalp 
In all experiments $5\%$ Gaussian random noise is added to the signal. \YZ{The MATLAB code for the following numerical examples is hosted on Github\footnote{https://github.com/lowrank/aem-isp}}.

\begin{figure}[!htb]
    \centering
    \includegraphics[scale=0.45]{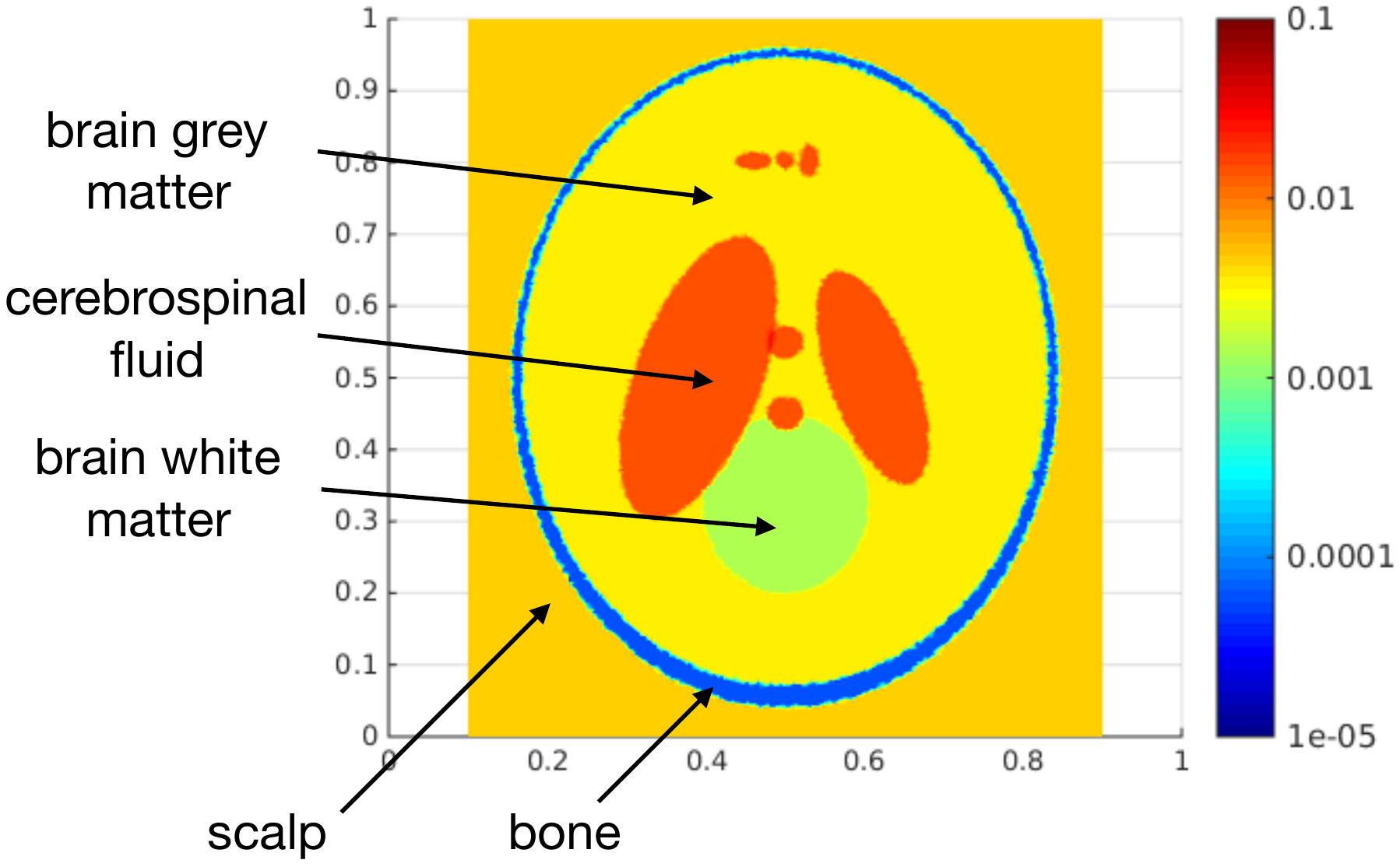}
    \caption{The conductivity function $\sigma_0$.}
    \label{conductivity}
\end{figure}

The alternating minimization algorithm is initialized with boundary sources $g_1, g_2$  of the form
\begin{equation}\label{eq:initial}
g_j(x) = \langle \cos\theta_j, \sin\theta_j \rangle \cdot n(x), \quad\theta_j\in [0,2\pi),\quad  j=1,2, \quad x\in\partial\Omega.
\end{equation}
When the difference between the angles $\theta_1$ and $\theta_2$ is relatively small, the resulting system is ill-conditioned. We thus take two different pairs of angles in the experiments: (i) $\theta_1 = 0$ and $\theta_2 = \frac{\pi}{2}$; (ii) $\theta_1 = \frac{5\pi}{6}$ and $\theta_2 = \pi$.

\subsubsection{Experiment 1: $\theta_1 = 0$ and $\theta_2 = \frac{\pi}{2}$}
The initial Neumann boundary conditions are
\begin{equation}  \label{eq:orthog}
g_1(x) = \langle 1, 0 \rangle \cdot n(x) \ , \quad\quad
g_2(x) = \langle 0, 1 \rangle \cdot n(x) \ .
\end{equation}
The choice is made to assess the performance of the alternating minimization algorithm when the gradients of $v_1$ and $v_2$ are nearly orthogonal. Note that these gradients are indeed orthogonal if $\sigma_0$ is constant.
The reconstructions are shown in Fig~\ref{case1}. The relative $L^2$ error is $2.99\%$ using the initial Neumann conditions $g_1, g_2$, and the relative $L^2$ error is $2.87\%$ using the Neumann conditions $g^*_1, g^*_2$ generated by the alternating minimization problem. The Neumann conditions $g_1, g_2, g^*_1, g^*_2$ are plotted in Fig~\ref{solutions}, where the horizontal axis represents the grid points on $\partial\Omega$. 
In this case, the initial guess $g_1, g_2$ is already very good and the optimization improves the result only to a small extent. 
\begin{figure}[!htb]
    \centering
    \includegraphics[width=0.3\textwidth, height=0.27\textwidth]{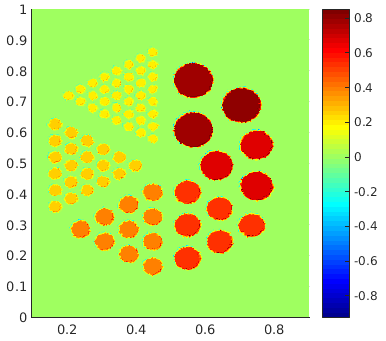}
    \includegraphics[width=0.3\textwidth, height=0.27\textwidth]{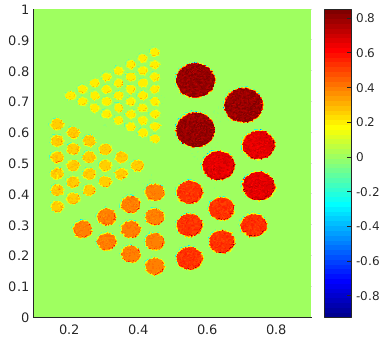}
    \includegraphics[width=0.27\textwidth, height=0.27\textwidth]{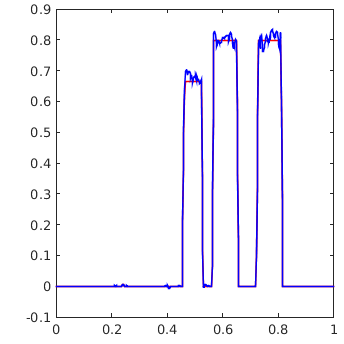}\\
    \includegraphics[width=0.3\textwidth, height=0.27\textwidth]{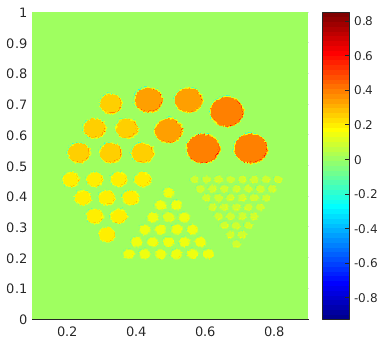}
    \includegraphics[width=0.3\textwidth, height=0.27\textwidth]{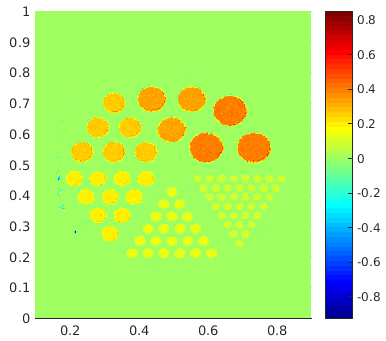}
    \includegraphics[width=0.27\textwidth, height=0.27\textwidth]{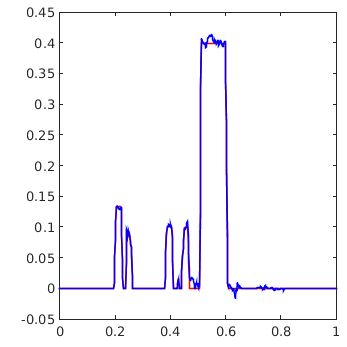}
    \caption{Experiment 1. Reconstruction with initial Neumann data corresponding to $\theta_1 = 0$ and $\theta_2 = \frac{\pi}{2}$ from full data. The first (second) row represents the first (second) component of the current density. From the left: the exact current density, the reconstructed current density, and the  exact (red) and reconstructed (blue) current density on the vertical line $x=0.6$.}
\label{case1}
\end{figure}

\begin{figure}[!htb]
    \centering
    \includegraphics[scale=0.48]{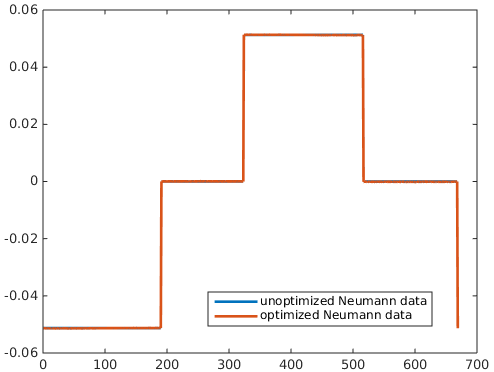} \includegraphics[scale=0.48]{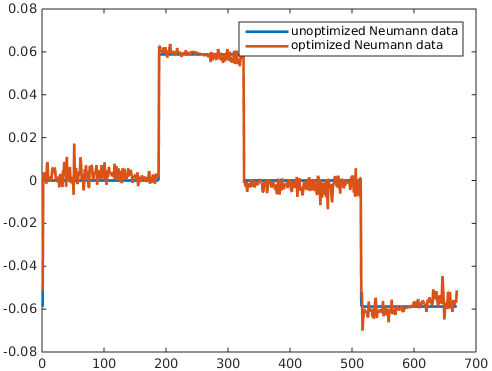} 
    \caption{Experiment 1. Initial Neumann data (corresponding to $\theta_1 = 0$ and $\theta_2 = \frac{\pi}{2}$) and optimized Neumann data. The graphs of $g_1$ and $g^*_1$ are on the left, the graphs of $g_2$ and $g^*_2$ are on the right, and the functions are plotted from the bottom left corner of the domain clockwise.}
  \label{solutions}
\end{figure}

\subsubsection{Experiment 2: $\theta_1 = \frac{5\pi}{6}$ and $\theta_2 = \pi$}
The initial Neumann boundary conditions are
\begin{equation} \label{eq:parag}
g_1(x) = \langle -\frac{\sqrt{3}}{2}, \frac{1}{2} \rangle \cdot n(x) \ , \quad\quad
g_2(x) = \langle -1, 0 \rangle \cdot n(x).
\end{equation}
This choice is made to assess the performance of the alternating minimization algorithm when the adjoint solutions $v_1$ and $v_2$ are nearly parallel. The reconstruction is shown in Fig~\ref{case2}. The relative $L^2$ error is $13.7\%$ using the initial boundary sources $g_1, g_2$, and the relative $L^2$ error is $6.10\%$ using the sources $g^*_1, g^*_2$ generated by the alternating minimization problem. The sources $g_1, g_2, g^*_1, g^*_2$ are plotted in Fig~\ref{solutions2}, where the horizontal axis represents the grid points on $\partial\Omega$. 
The alternating minimization improves the result significantly in this case.
\begin{figure}[!htb]
    \centering
    \includegraphics[width=0.3\textwidth, height=0.27\textwidth]{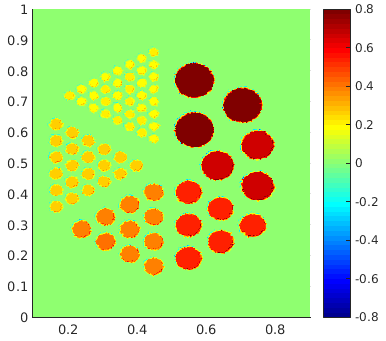}
    \includegraphics[width=0.3\textwidth, height=0.27\textwidth]{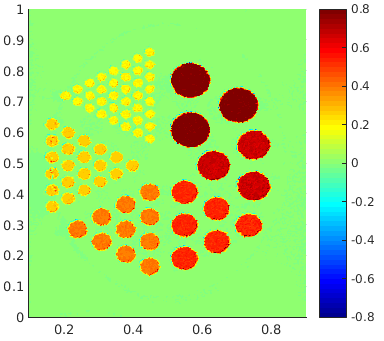}
    \includegraphics[width=0.27\textwidth, height=0.27\textwidth]{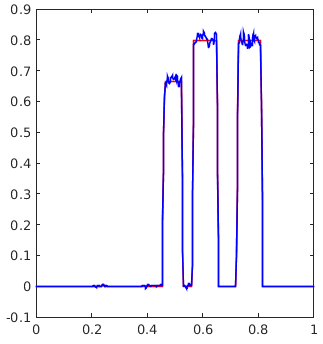}\\
    \includegraphics[width=0.3\textwidth, height=0.27\textwidth]{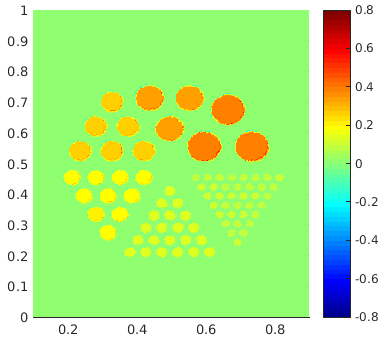}
    \includegraphics[width=0.3\textwidth, height=0.27\textwidth]{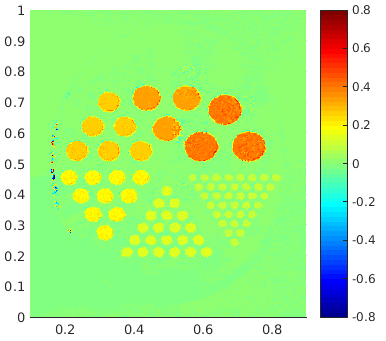}
    \includegraphics[width=0.27\textwidth, height=0.27\textwidth]{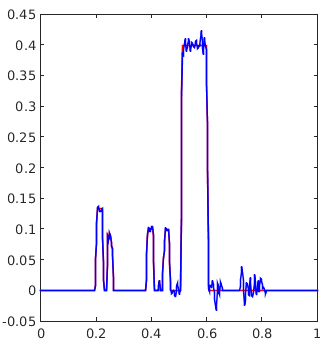}
    \caption{Experiment 2. Reconstruction with initial Neumann data corresponding to $\theta_1 = \frac{5\pi}{6}$ and $\theta_2 = \pi$ from full data. The first (second) row represents the first (second) component of the current density. From the left: the exact current density, the reconstructed current density, and the  exact (red) and reconstructed (blue) current density on the vertical line $x=0.6$.}
\label{case2}
\end{figure}

\begin{figure}[!htb]
    \centering
    \includegraphics[scale=0.48]{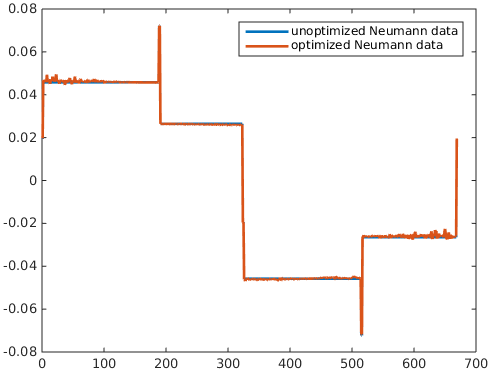} \includegraphics[scale=0.48]{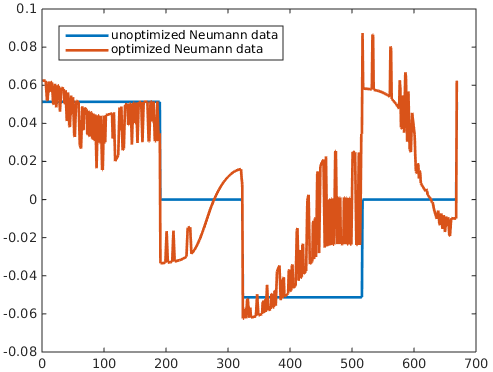} 
    \caption{Experiment 2. Initial Neumann data  (corresponding to $\theta_1 = \frac{5\pi}{6}$ and $\theta_2 = \pi$) and optimized Neumann data. The graphs of $g_1$ and $g^*_1$ are on the left, the graphs of $g_2$ and $g^*_2$ are on the right, and the functions are plotted from the bottom left corner of the domain clockwise.}
    \label{solutions2}
\end{figure}

\subsection{Partial Data}

We also tested the reconstruction algorithm for the case of partial boundary measurements, where measurements are only taken on a part of the boundary $\Gamma\subset \partial\Omega$. In this case, one can only prescribe Neumann conditions $g_1, g_2$ that are compactly supported in the interior of $\Gamma$.
It is not possible to find $v_j$, $j=1,\cdots,n$, whose gradients are uniformly linearly independent in $\Omega$ 
since the gradients are linearly dependent on the boundary. However, since $\BJ$ is compactly supported in $\Omega$, we can look for $v_j$, $j=1,\cdots,n$ whose gradients are uniformly linearly independent on $\text{supp}(\BJ)$. If such $v_j$ exist, the internal functional~\eqref{eq:internal} on $\text{supp}(\BJ)$ is available from the measurements, and the reconstruction procedure works identically from this point on.

In the following examples, measurements on the bottom edge of $\Omega = [0.1, 0.9]\times [0, 1]$ are absent, that is, $\partial\Omega\backslash\Gamma = [0.1, 0.9]\times \{0\}$. The boundary sources $g_1, g_2$ are the same as in \eqref{eq:initial}, except that they vanish on the bottom edge. We will again consider the two pairs of angles: (i) $\theta_1 = 0$ and $\theta_2 = \frac{\pi}{2}$; (ii) $\theta_1 = \frac{5\pi}{6}$ and $\theta_2 = \pi$.

\subsubsection{Experiment 3: $\theta_1 = 0$ and $\theta_2 = \frac{\pi}{2}$}

The initial boundary sources $g_1, g_2$ are as in~\eqref{eq:orthog} on $\Gamma$, and are set to be zero on $\partial\Omega\backslash\Gamma$. Note that the gradients of $v_1, v_2$ cannot be everywhere orthogonal due to the boundary constraint that $\frac{\partial v_1}{\partial n} = \frac{\partial v_2}{\partial n} = 0$ on the bottom boundary.
The reconstruction is shown in Fig~\ref{case3}. The relative $L^2$ error is $2.82\%$ using the initial  $g_1, g_2$, and the relative $L^2$ error is $2.70\%$ using the boundary sources $g^*_1, g^*_2$ generated by the alternating minimization problem. The boundary sources $g_1, g_2, g^*_1, g^*_2$ are plotted in Fig~\ref{solutions3}, where the horizontal axis represents the grid points on $\partial\Omega$.

\begin{figure}[!htb]
    \centering
    \includegraphics[width=0.3\textwidth, height=0.27\textwidth]{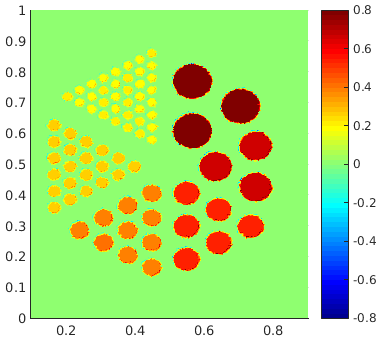}
    \includegraphics[width=0.3\textwidth, height=0.27\textwidth]{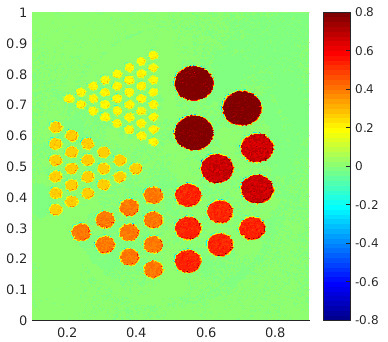}
    \includegraphics[width=0.27\textwidth, height=0.27\textwidth]{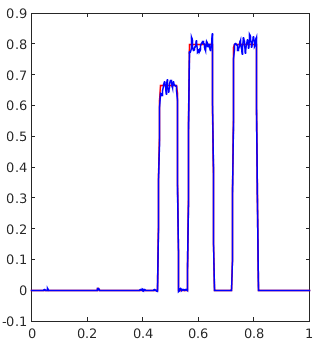}\\
    \includegraphics[width=0.3\textwidth, height=0.27\textwidth]{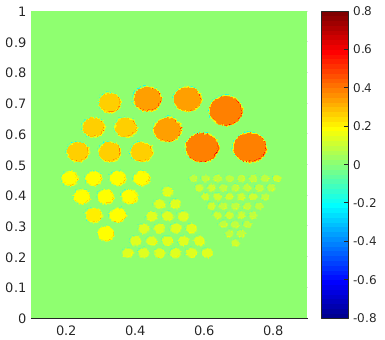}
    \includegraphics[width=0.3\textwidth, height=0.27\textwidth]{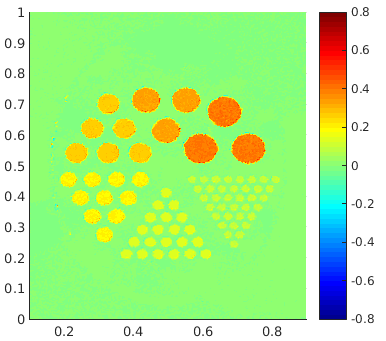}
    \includegraphics[width=0.27\textwidth, height=0.27\textwidth]{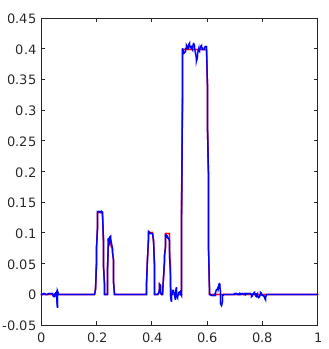}
    \caption{Experiment 3. Reconstruction with initial Neumann data corresponding to $\theta_1 = 0$ and $\theta_2 = \frac{\pi}{2}$ from partial data. The first (second) row represents the first (second) component of the current density. From the left: the exact current density, the reconstructed current density, and the  exact (red) and reconstructed (blue) current density on the vertical line $x=0.6$.}
\label{case3}
\end{figure}

\begin{figure}[!htb]
    \centering
    \includegraphics[scale=0.48]{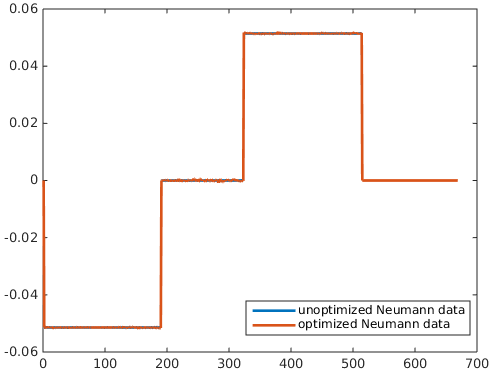} \includegraphics[scale=0.48]{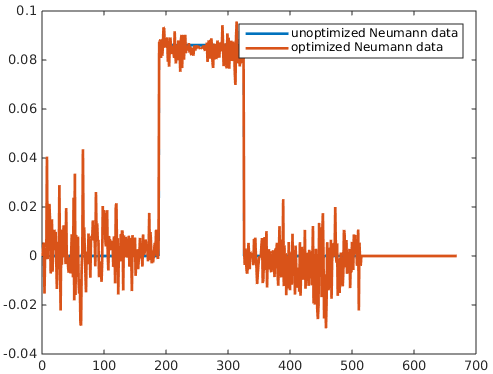} 
    \caption{Experiment 3. Initial Neumann conditions (corresponding to $\theta_1 = 0$ and $\theta_2 = \frac{\pi}{2}$) and optimized Neumann conditions with partial data. The graphs of $g_1$ and $g^*_1$ are on the left, the graphs of $g_2$ and $g^*_2$ are on the right, and the functions are plotted from the bottom left corner of the domain clockwise.}
    \label{solutions3}
\end{figure}

\subsubsection{Experiment 4: $\theta_1 = \frac{5\pi}{6}$ and $\theta_2 = \pi$}
The initial Neumann boundary conditions $g_1, g_2$ are as in~\eqref{eq:parag} on $\Gamma$, and are set to be zero on $\partial\Omega\backslash\Gamma$. The reconstruction is shown in Fig~\ref{case4}. The relative $L^2$ error is $34.1\%$ using the initial boundary sources $g_1, g_2$, and the relative $L^2$ error is $10.2\%$ using the sources $g^*_1, g^*_2$ generated by the alternating minimization problem. The sources $g_1, g_2, g^*_1, g^*_2$ are plotted in Fig~\ref{solutions4}, where the horizontal axis represents the grid points on $\partial\Omega$. Evidently, optimization improves the result significantly.

\begin{figure}[!htb]
    \centering
    \includegraphics[width=0.3\textwidth, height=0.27\textwidth]{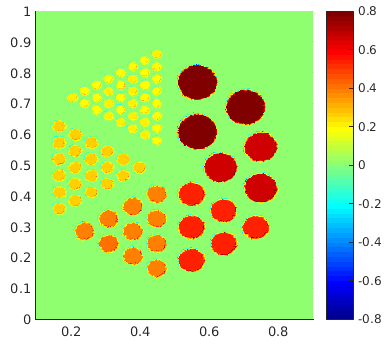}
    \includegraphics[width=0.3\textwidth, height=0.27\textwidth]{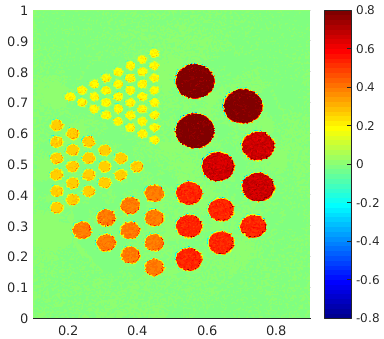}
    \includegraphics[width=0.27\textwidth, height=0.27\textwidth]{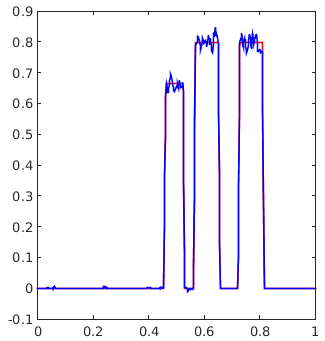}\\
    \includegraphics[width=0.3\textwidth, height=0.27\textwidth]{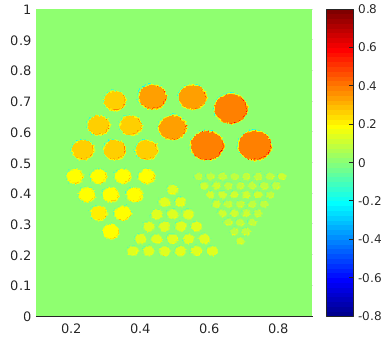}
    \includegraphics[width=0.3\textwidth, height=0.27\textwidth]{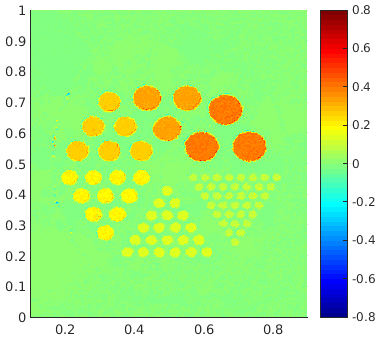}
    \includegraphics[width=0.27\textwidth, height=0.27\textwidth]{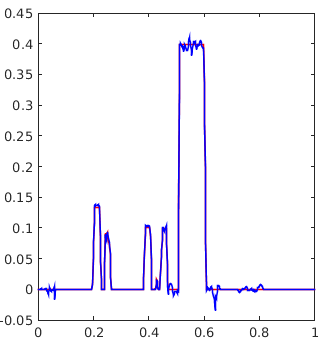}
    \caption{Experiment 4. Reconstruction with initial Neumann data corresponding to $\theta_1 = \frac{5\pi}{6}$ and $\theta_2 = \pi$ from partial data. The first (second) row represents the first (second) component of the current density. From the left: the exact current density, the reconstructed current density, and the exact (red) and reconstructed (blue) current density on the vertical line $x=0.6$.}
\label{case4}
\end{figure}

\begin{figure}[!htb]
    \centering
    \includegraphics[scale=0.479]{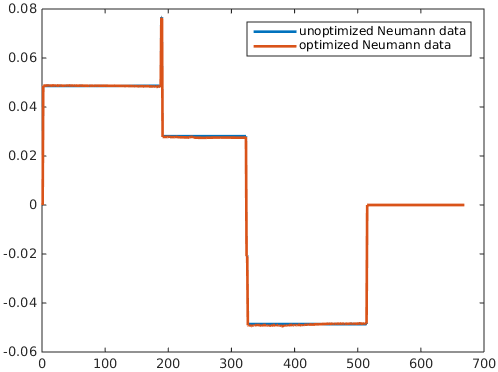} \includegraphics[scale=0.479]{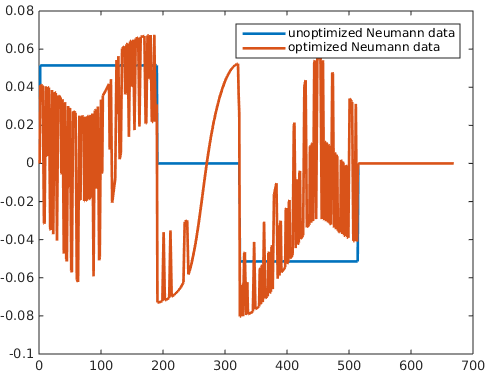} 
    \caption{Experiment 4. Initial Neumann data (corresponding to $\theta_1 = \frac{5\pi}{6}$ and $\theta_2 = \pi$) and optimized Neumann data with partial data. The graphs of $g_1$ and $g^*_1$ are on the left, the graphs of $g_2$ and $g^*_2$ are on the right, and the functions are plotted from the bottom left corner of the domain clockwise.}
    \label{solutions4}
\end{figure}

\section{Discussion}
\label{sec:conclusion}
In this paper, we formulated a mathematical model of an acoustically-modulated electrical source problem. We showed that boundary measurement of the electric potential in the presence of acoustic modulation leads to knowledge of an internal functional. Based on this observation, we devised explicit procedures to reconstruct the (unmodulated) source current $\BJ_0$ from the internal functional. The reconstruction is shown to be unique with Lipschitz stability, which serves as the mathematical justification for the elimination of non-uniqueness in the classical inverse problem. We also present numerical implementations of the proposed procedures with both full and partial boundary measurement, including an alternating minimization algorithm that improves numerical stability. We note that the model we consider holds for the case of steady currents, and is suitable for applications to low-frequency biological currents. Although this work was motivated by neurophysiologic applications, similar considerations apply to cardiac electrophysiology. In future work, we intend to explore the high-frequency regime where it is necessary to employ the apparatus of the full Maxwell system.

\section*{Acknowledgments}
JCS is indebted to his father, Donald L. Schotland, M.D., who introduced him to the field of neurophysiology. This paper is dedicated to his memory.  He would also like to acknowledge the influence of Michael J. O'Connor, M.D. for stimulating his interest in the surgical treatment of epilepsy. The work of JCS was supported in part by the NSF grant DMS-1912821 and the AFOSR grant FA9550-19-1-0320. The work of YY was supported in part by the NSF grants DMS-1715178 and DMS-2006881. 

\bibliographystyle{siam}
\bibliography{main.bib}
%
%
%
%
%
%
%

\end{document}